\documentclass{amsart}
\usepackage{hyperref}

\newcommand*{\mailto}[1]{\href{mailto:#1}{\nolinkurl{#1}}}

\newtheorem{theorem}{Theorem}[section]

\newtheorem{lemma}[theorem]{Lemma}
\newtheorem{corollary}[theorem]{Corollary}

\newcommand{\R}{{\mathbb R}}
\newcommand{\N}{{\mathbb N}}
\newcommand{\Z}{{\mathbb Z}}
\newcommand{\C}{{\mathbb C}}

\newcommand{\nn}{\nonumber}
\newcommand{\be}{\begin{equation}}
\newcommand{\ee}{\end{equation}}

\newcommand{\ti}{\tilde}

\newcommand{\floor}[1]{\lfloor#1 \rfloor}
\newcommand{\ceil}[1]{\lceil#1 \rceil}

\newcommand{\eps}{\varepsilon}

\newcommand{\sig}{\sigma}
\newcommand{\lam}{\lambda}

\numberwithin{equation}{section}

\begin{document}

\title{Relative Oscillation Theory for Jacobi Matrices}

\author[K. Ammann]{Kerstin Ammann}
\address{Faculty of Mathematics\\
Nordbergstrasse 15\\ 1090 Wien\\ Austria}
\email{\mailto{Kerstin.Ammann@univie.ac.at}}
\urladdr{\url{http://www.mat.univie.ac.at/~kerstin/}}

\author[G. Teschl]{Gerald Teschl}
\address{Faculty of Mathematics\\
Nordbergstrasse 15\\ 1090 Wien\\ Austria\\ and International Erwin Schr\"odinger
Institute for Mathematical Physics, Boltzmanngasse 9\\ 1090 Wien\\ Austria}
\email{\mailto{Gerald.Teschl@univie.ac.at}}
\urladdr{\url{http://www.mat.univie.ac.at/~gerald/}}

\thanks{{\it in Proceedings of the 14th International Conference on Difference Equations and Applications,
M. Bohner (ed) et al., 105--115, U\u{g}ur--Bah\c{c}e\c{s}ehir University Publishing Company, Istanbul, 2009.}}
\thanks{{\it Research supported by the Austrian Science Fund (FWF) under Grant No.\ Y330}}


\begin{abstract}
We develop relative oscillation theory for Jacobi matrices
which, rather than counting the number of eigenvalues of one single matrix, counts the
difference between the number of eigenvalues of two different matrices.
This is done by replacing nodes of solutions associated with one matrix by weighted nodes of Wronskians of
solutions of two different matrices.
\end{abstract}

\maketitle

\section{Introduction}
\label{sec:int}

Oscillation theory for second-order differential and difference equations has a long tradition
originating in the seminal work of Sturm from 1836 \cite{stu}. Since then the subject is
continuously growing and many monographs have been devoted entirely to this subject.
The most recent one being the monumental treatise by Agarwal, Bohner, Grace, and O'Regan
\cite{abgr}. One of the key results of classical oscillation theory is the fact, the $k$'th eigenfunction
has precisely $k-1$ nodes (i.e., sign flips) and for a suitably chosen solution of the underlying difference
equation, the number of nodes of this solutions equals the number of eigenvalues below a given
value. Our aim is add a new wrinkle to this classical result by showing that the number of
weighted nodes of the Wronskian (also known as Casoratian) of two suitable solutions of two
different Jacobi difference equations can be used to count the difference between the number of
eigenvalues of the two associated Jacobi matrices.

That Wronskians are related to oscillation theory is indicated by an old paper of Leighton \cite{lei},
who noted that if two solutions have a non-vanishing Wronskian, then their zeros must intertwine each other.
However, it seems their real power was realized only later by Gesztesy, Simon, and Teschl in \cite{gst} with
the corresponding extension to Jacobi operators given by Teschl \cite{tosc}. For a pedagogical discussion
we refer to the survey by Simon \cite{sim}. That these results are just
the tip of the iceberg was discovered only recently by Kr\"uger and Teschl \cite{kt}, \cite{kt2}, \cite{kt3}.
Our result generalizes the main result for the case of Sturm--Liouville operators from \cite{kt} 
to the case of Jacobi matrices. 

To set the stage, let us fix some real numbers $a(j)<0$, $b(j)$, $j=1,\cdots,N-1$ and consider the
Jacobi matrix
\be
H = \left(
\begin{matrix}
b(1)	& a(1) 		& 0		  		& 0		  		& 0 			\\
a(1)	& b(2) 		& \ddots  		& 0		  		& 0 			\\
0 			& \ddots	  	& \ddots    	& \ddots 		& 0 			\\
0			& 0 	    	& a(N - 1) 	& b(N - 2) 	& a(N - 2) 	\\
0			& 0				& 0				& a(N - 2)	& b(N - 1) 	\\
\end{matrix}\right).
\ee
in the Hilbert space $\C^{N-1}$. Furthermore, let $s_\pm(z,n)$ be the solutions of the underlying difference
equation (set $a(0)=a(N-1)=a(N)=-1$, $b(N)=0$)
\be\label{jde}
a(n) u(n+1) + b(n) u(n) + a(n-1) u(n-1) = z u(n), \qquad n=1, \dots, N,
\ee
corresponding to the initial conditions
\be
s_-(z,0)=0, s_-(z,1)=1, \qquad s_+(z,N)=0, s_+(z,N+1)=1.
\ee
Note that $s_-(\lam,n)$ (resp.\ $s_+(\lam,n)$) will be an eigenvector of $H$ corresponding to the eigenvalue
$\lam\in\R$ if and only if $s_-(\lam,N)=0$ (resp.\ $s_+(\lam,0)=0$). We will abbreviate $s(z,n)=s_-(z,n)$.

We call $n$ a node of a solution $u$ of \eqref{jde} if either 
\be
u(n) = 0 \quad\text{ or }\quad u(n) u(n + 1) < 0.
\ee
We say that a node $n_0$ of $u$ \emph{lies between $m$ and $n$} if either 
\be
m < n_0 < n \quad\text{ or }\quad n_0 = m \text{ but } u(m) \ne 0.
\ee
$\#_{(m,n)}(u)$ denotes the number of nodes of $u$ between $m$ and $n$ and $\#(u) = \#_{(0,N)}(u)$.

Then we have the following classical result alluded to before (see e.g., \cite{at}, \cite{tjac}):

\begin{theorem}
Let $H$ be a Jacobi matrix and $s(z,n)$ a corresponding solution of the underlying difference
equation \eqref{jde} corresponding to the initial condition $s(z,0)=0$. Then for every $\lam\in\R$ the number of nodes of
$s(\lam,n)$ equals the number of eigenvalues of $H$ below $\lam$:
\be
\#(s(\lam)) = \# \{E \in \sig(H) | E < \lam \}.
\ee
Here $\sig(H)$ denotes the spectrum of $H$, that is, the set of eigenvalues.
\end{theorem}

To generalize this result we will now consider two Jacobi matrices $H_0$ and $H_1$ associated with the coefficients
$a_0(n)=a_1(n) \equiv a(n)$ and $b_0(n)$ respectively $b_1(n)$. The corresponding solutions will be denoted by
$s_{j,\pm}(n)$, $j=0,1$, in obvious notation. Given two solutions $u_j$ of the difference equations associated with
$H_j$ we denote by
\be
W_n(u_0,u_1)= a(n)(u_0(n) u_1(n+1) - u_0(n+1) u_1(n))
\ee
their Wronskian. As already anticipated we will relate the number of nodes of such Wronskians to the difference between
the eigenvalues of $H_0$ and $H_1$. Since this difference is a signed quantity, we will need to weight the nodes according
to the sign of the difference between $H_0$ and $H_1$ as follows: Set
\be
\#_n(u_0,u_1) = \begin{cases}
	& \text{if } b_0(n+1) - b_1(n+1) > 0 \text{ and } \\
1,	& \text{either } W_n(u_0,u_1) W_{n+1}(u_0,u_1) < 0 \\ 
	& \text{or } W_n(u_0,u_1) = 0 \text{ and } W_{n+1}(u_0,u_1) \neq 0, \\[2mm]
	& \text{if } b_0(n+1) - b_1(n+1) < 0 \text{ and } \\
-1,	& \text{either } W_n(u_0,u_1) W_{n+1}(u_0,u_1) < 0 \\
	& \text{or } W_n(u_0,u_1) \neq 0 \text{ and } W_{n+1}(u_0,u_1) = 0, \\[2mm]
0,	& \text{otherwise}.
\end{cases}
\ee
Then we say the Wronskian has a weighted node at $n$ if $\#_n(u_0,u_1) \neq 0$.
The number of weighted nodes of the Wronskian between $0$ and $N$ is denoted as
\be
\#(u_0, u_1) = \sum_{j=0}^{N-1} \#_j(u_0,u_1) -
\begin{cases}
0, & \text{if } W_{0}(u_0,u_1) \neq 0,\\
1, & \text{if } W_{0}(u_0,u_1) = 0.
\end{cases}
\ee
With this notation our main result reads

\begin{theorem}\label{thm:main}
Let $H_0$, $H_1$ be two Jacobi matrices with $a_0=a_1$ and $s_{j,\pm}(z,n)$, $j=0,1$, the corresponding solutions
of the underlying difference equations. Then for $\lam_0,\lam_1\in\R$ the number of weighted nodes of
$W(s_{0,-}(\lam_0),s_{1,+}(\lam_1))$ equals the number of eigenvalues of $H_1$ below $\lam_1$ minus the
number of eigenvalues of $H_0$ below or equal to $\lam_0$:
\begin{align}\nn
&\#(s_{0,-}(\lam_0),s_{1,+}(\lam_1)) = \#(s_{0,+}(\lam_0),s_{1,-}(\lam_1)) =\\
& \qquad  = \# \{E \in \sig(H_1) | E < \lam_1\} - \# \{E \in \sig(H_0) | E\leq \lam_0 \}.
\end{align}
Here $\sig(H)$ denotes the spectrum of $H$, that is, the set of eigenvalues.
\end{theorem}

The proof is based on Pr\"ufer angles to be investigated in Section~\ref{sec:pa}. It
will be given in Section~\ref{sec:proof}.

An extension to Jacobi operators on $\N$ respectively $\Z$ is in preparation.

\section{Pr\"ufer angles}
\label{sec:pa}

Since any nontrivial solution of \eqref{jde} cannot vanish at two consecutive points we can introduce Pr\"ufer variables
$(\rho_u(n),\theta_u(n))$ in the usual way (cf., e.g., \cite[Chap.~4]{tjac}) via
\be
u(n) = \rho_u(n) \sin(\theta_u(n)), \qquad
u(n + 1) = \rho_u(n) \cos(\theta_u(n)).
\ee
Note that $\rho_u(n) > 0$ for all $n \in \Z$ and $\theta_u(n)$ is only defined up to an additive integer multiple of $2 \pi$,
depending on $n$. For our further investigations it is essential to gain unique values for the Pr\"ufer angle and therefore
we fix $\theta_u(0)$ and require 
\be \label{normalth}
\ceil{\theta_u(n)/ \pi} \le \ceil{\theta_u(n + 1)/ \pi} \le \ceil{\theta_u(n)/ \pi} + 1,
\ee
where $\ceil{x}= \min \{n \in \Z \,|\, n \geq x \}$ denotes the usual ceiling function. Then the following easy result is well-known.

\begin{lemma} \label{gammaGammaint}
Define $k$, $\gamma$, $\Gamma$ via
\be \label{thetaNotation}
\theta_u(n) = k \pi + \gamma, \quad
\theta_u(n + 1) = k \pi + \Gamma, \qquad 
\gamma \in (0,\pi],\: \Gamma \in (0,2 \pi],\: k\in\Z.
\ee
Then
\be 
\gamma \in
\begin{cases}
(0, \frac{\pi}{2}] & \text{iff $n$ is not a node},\\
(\frac{\pi}{2}, \pi] & \text{iff $n$ is a node},
\end{cases}
\ee
and
\be
\Gamma \in
\begin{cases}
(0, \pi] & \text{iff $n$ is not a node},\\
(\pi, 2\pi) &\text{iff $n$ is a node}.
\end{cases}
\ee
Moreover,
\be
\theta_u(n) = k \pi + \frac{\pi}{2} \quad\Leftrightarrow\quad \theta_u(n + 1) = (k + 1) \pi.
\ee
\end{lemma}

As a consequence we obtain

\begin{corollary}
We have
\be
\ceil{\frac{\theta_u(n+1)}{\pi}} = 
\begin{cases}
\ceil{\frac{\theta_u(n)}{\pi}} + 1 & \text{if } $n$ \text{ is a node}, \\
\ceil{\frac{\theta_u(n)}{\pi}} & \text{otherwise}.
\end{cases}
\ee
\end{corollary}

In particular, we obtain
\be
\#(u) = \ceil{\frac{\theta_u(N)}{\pi}} - \floor{\frac{\theta_u(0)}{\pi}} - 1,
\ee
where $\floor{x}= \max \{n \in \Z \,|\, n \leq x \}$ is the usual floor function.

To find the analogous formula for the number of weighted nodes of a Wronskian we observe
\be\label{wronskipruefb}
W_n(u_0, u_1) = - a(n) \rho_{u_0}(n) \rho_{u_1}(n) \sin(\Delta_{u_0, u_1}(n)),
\ee
where
\be
\Delta_{u_0, u_1}(n) = \theta_{u_1}(n) - \theta_{u_0}(n).
\ee
Furthermore, note
\be \label{gfb1step}
W_{n + 1}(u_0, u_1) - W_{n}(u_0, u_1) = (b_0(n + 1) - b_1(n + 1)) u_0(n + 1) u_1(n + 1).
\ee
As a straightforward consequence of Lemma~\ref{gammaGammaint} we obtain

\begin{lemma} \label{gammaGammaint2}
Fix some $n$ and let $\theta_j(n) = k_j \pi + \gamma_j$ with $\gamma_j \in (0, \pi]$ and $\theta_j(n + 1) = k_j \pi + \Gamma_j$
with $\Gamma_j \in (0,2 \pi]$ for $j=0,1$. Then we have 
\be
\Delta_{u_0,u_1}(n) = (k_1 - k_0) \pi + \gamma_1 - \gamma_0 \quad\text{ and }\quad
\Delta_{u_0,u_1}(n+1) = (k_1 - k_0) \pi + \Gamma_1 - \Gamma_0,
\ee
where
\begin{description}
\item [(1)] either $u_0$ and $u_1$ have a node at $n$ or both do not have a node at $n$, then
\be
\gamma_1 - \gamma_0 \in (-\frac{\pi}{2}, \frac{\pi}{2}) \quad\text{ and }\quad \Gamma_1 - \Gamma_0 \in (-\pi,\pi).
\ee
\item [(2)] $u_1$ has no node at $n$, but $u_0$ has a node at $n$, then
\be
\gamma_1 - \gamma_0 \in (-\pi,0) \quad\text{ and }\quad \Gamma_1 - \Gamma_0 \in (-2\pi,0).
\ee
\item [(3)] $u_1$ has a node at $n$, but $u_0$ has no node at $n$, then
\be
\gamma_1 - \gamma_0 \in (0,\pi) \quad\text{ and }\quad \Gamma_1 - \Gamma_0 \in (0, 2\pi).
\ee
\end{description}
\end{lemma}

Now we are able to show

\begin{lemma}
Fix some $n$. Then, if $b_0(n+1) \geq b_1(n+1)$, we have
\be \label{DeltarelationPos}
\ceil{\Delta_{u_0,u_1}(n)/ \pi} \leq \ceil{\Delta_{u_0,u_1}(n + 1)/ \pi} \leq \ceil{\Delta_{u_0,u_1}(n)/ \pi} + 1
\ee
and if $b_0(n+1) \leq b_1(n+1)$, we have
\be \label{DeltarelationNeg}
\ceil{\Delta_{u_0,u_1}(n)/ \pi} -1 \leq \ceil{\Delta_{u_0,u_1}(n + 1)/ \pi} \leq \ceil{\Delta_{u_0,u_1}(n)/ \pi}.
\ee
\end{lemma}

\begin{proof}
We will use the notation from Lemma~\ref{gammaGammaint2} where we assume $k_0=k_1=0$ without loss of
generality. In particular, Lemma~\ref{gammaGammaint2} implies
\[
\ceil{\Delta_{u_0,u_1}(n)/ \pi} -1 \leq \ceil{\Delta_{u_0,u_1}(n + 1)/ \pi} \leq \ceil{\Delta_{u_0,u_1}(n)/ \pi} + 1.
\]
Hence, to show \eqref{DeltarelationPos} there are two cases to exclude. Namely, (i) $\Delta_{u_0,u_1}(n)\in(0,\frac{\pi}{2})$,
$\Delta_{u_0,u_1}(n+1)\in(-\pi,0]$ (from case (1)) and (ii) $\Delta_{u_0,u_1}(n)\in(-\pi,0)$, $\Delta_{u_0,u_1}(n+1)\in(-2\pi,-\pi]$ (from case (2)).
But in case (i) we obtain a contradiction from \eqref{gfb1step}:
\[
\underbrace{W_{n+1}(u_0,u_1)}_{\le 0} = \underbrace{W_{n}(u_0,u_1)}_{>0}
+ \underbrace{(b_0(n + 1) - b_1(n + 1))}_{\ge 0} \underbrace{u_0(n + 1) u_1(n + 1)}_{\ge 0}.
\]
Similarly, in case (ii) equation \eqref{gfb1step} implies
\[
\underbrace{W_{n+1}(u_0,u_1)}_{\ge 0} = \underbrace{W_{n}(u_0,u_1)}_{<0}
+ \underbrace{(b_0(n + 1) - b_1(n + 1))}_{\ge 0} \underbrace{u_0(n + 1) u_1(n + 1)}_{\le 0}.
\]
Equation \eqref{DeltarelationNeg} can be established in a similar fashion.
\end{proof}

\begin{lemma}
Let $n \in \Z$, then
\begin{description}
\item[(1)] $W_n(u_0,u_1) = W_{n+1}(u_0,u_1) = 0$ or $W_n(u_0,u_1) W_{n+1}(u_0,u_1) > 0$ implies
\be 
\ceil{\frac{\Delta_{u_0,u_1}(n+1)}{\pi}} = \ceil{\frac{\Delta_{u_0,u_1}(n)}{\pi}}.
\ee
\item[(2)] $W_n(u_0,u_1) W_{n+1}(u_0,u_1) < 0$ implies 
\be
\ceil{\frac{\Delta_{u_0,u_1}(n+1)}{\pi}} = 
\begin{cases}
\ceil{\frac{\Delta_{u_0,u_1}(n)}{\pi}} + 1, & \text{if } b_0(n+1) > b_1(n+1),\\
\ceil{\frac{\Delta_{u_0,u_1}(n)}{\pi}} - 1, & \text{if } b_0(n+1) < b_1(n+1).
\end{cases}
\ee
\item[(3)] $W_n(u_0,u_1) = 0$ and $W_{n+1}(u_0,u_1) \neq 0$ implies 
\be
\ceil{\frac{\Delta_{u_0,u_1}(n+1)}{\pi}} = 
\begin{cases}
\ceil{\frac{\Delta_{u_0,u_1}(n)}{\pi}}+1, & \text{if } b_0(n+1) > b_1(n+1),\\
\ceil{\frac{\Delta_{u_0,u_1}(n)}{\pi}}, & \text{if } b_0(n+1) < b_1(n+1).
\end{cases}
\ee
\item[(4)] $W_n(u_0,u_1) \neq 0$ and $W_{n+1}(u_0,u_1) = 0$ implies
\be
\ceil{\frac{\Delta_{u_0,u_1}(n+1)}{\pi}} =
\begin{cases}
\ceil{\frac{\Delta_{u_0,u_1}(n)}{\pi}}, & \text{if } b_0(n+1) > b_1(n+1),\\
\ceil{\frac{\Delta_{u_0,u_1}(n)}{\pi}}-1, & \text{if } b_0(n+1) < b_1(n+1).
\end{cases}
\ee
\end{description}
Note that in the cases (2)--(4) we necessarily have $b_0(n+1) \neq b_1(n+1)$. 
\end{lemma}

\begin{proof}
We will use the notation from Lemma~\ref{gammaGammaint2} where we assume $k_0=k_1=0$ without loss of
generality. Moreover, interchanging $u_0$ and $u_1$ using $\Delta_{u_1,u_0}= -\Delta_{u_0,u_1}(n)$ and
\[
\ceil{-x} = \begin{cases}
- \ceil{x} & \text{if } x \in \Z, \\
- \ceil{x} + 1 & \text{otherwise},
\end{cases}
\]
we see that it suffices to show one case $b_0(n+1) \geq b_1(n+1)$ or $b_0(n+1) \leq b_1(n+1)$.

Suppose $W_n(u_0,u_1) = W_{n+1}(u_0,u_1) = 0$ and $W_n(u_0,u_1) W_{n+1}(u_0,u_1) > 0$ do not hold,
then by \eqref{gfb1step} we have
\[
W_{n+1}(u_0,u_1) - W_n(u_0,u_1) = (b_0(n+1) - b_1(n+1)) u_0(n+1) u_1(n+1) \neq 0
\]
and hence $b_0(n+1) \neq b_1(n+1)$.

(1) and (2). Suppose $W_n(u_0,u_1) = W_{n+1}(u_0,u_1) = 0$, then by \eqref{wronskipruefb} we infer
\[
\sin(\Delta_{u_0,u_1}(n)) = \sin(\gamma_1 - \gamma_0) = 0, \quad
 \sin(\Delta_{u_0,u_1}(n+1)) = \sin(\Gamma_1 - \Gamma_0) = 0,
\]
where $\gamma_0$, $\gamma_1 \in (0,\pi]$. Thus $\gamma_0 = \gamma_1$ and we have case (1) of
Lemma~\ref{gammaGammaint2} which implies $\Gamma_1 - \Gamma_0 \in (-\pi,\pi)$ and we conclude
$\Gamma_1 - \Gamma_0 = 0$. In summary, $\Delta_{u_0,u_1}(n) = \Delta_{u_0,u_1}(n+1) = 0$ as claimed.

Next suppose $W_n(u_0,u_1) W_{n+1}(u_0,u_1) \neq 0$, then by \eqref{wronskipruefb} the sign of the Wronskian
at $n$ equals the sign of $\sin(\Delta_{u_0,u_1}(n))$ and hence \eqref{DeltarelationPos} respectively
\eqref{DeltarelationNeg} finish the proof of case (1) and (2).

(3). By \eqref{wronskipruefb} we conclude $\Delta_{u_0,u_1}(n) = \gamma_1 - \gamma_0 \equiv 0 \mod \pi$,
where $\gamma_0$, $\gamma_1 \in (0,\pi]$ and thus $\gamma_1 - \gamma_0 = 0$. So we have case (1) of
Lemma~\ref{gammaGammaint2} and hence $\Delta_{u_0,u_1}(n+1) = \Gamma_1 - \Gamma_0 \in (-\pi,\pi)$. That is,
\[
\ceil{\Delta_{u_0,u_1}(n)/ \pi} \leq \ceil{\Delta_{u_0,u_1}(n + 1)/ \pi} \leq \ceil{\Delta_{u_0,u_1}(n)/ \pi}+1
\]
and \eqref{DeltarelationNeg} finishes the proof of case (3) for $b_0(n+1) < b_1(n+1)$.

(4). By \eqref{wronskipruefb} we have $\Delta_{u_0,u_1}(n+1) = \Gamma_1 - \Gamma_0 \equiv 0 \mod \pi$ and
Lemma~\ref{gammaGammaint2} leaves us with the following possibilities
\[
\begin{array}{lllll}
\text{(a)} & \Delta_{u_0,u_1}(n) \in (-\frac{\pi}{2}, \frac{\pi}{2}) \ & \text{ and }\ & \Delta_{u_0,u_1}(n+1) = 0,\\
\text{(b)} & \Delta_{u_0,u_1}(n) \in (-\pi,0) \ & \text{ and }\ & \Delta_{u_0,u_1}(n+1) = - \pi,\\
\text{(c)} & \Delta_{u_0,u_1}(n) \in (0,\pi) \ & \text{ and }\ & \Delta_{u_0,u_1}(n+1) = \pi.\\
\end{array}
\]
and \eqref{DeltarelationPos} shows (4) if $b_0(n+1) > b_1(n+1)$.
\end{proof}

As a consequence we obtain the desired formula
\be\label{wronskinodes}
\#(u_0, u_1) = \ceil{\Delta_{u_0, u_1}(N)/\pi} - \floor{\Delta_{u_0, u_1}(0)/ \pi} - 1.
\ee

\section{Proof of the main theorem}
\label{sec:proof}

Our strategy will be to interpolate between $H_0$ and $H_1$ using $H_\eps = (1-\eps)H_0 + \eps H_1$, that is,
$a_\eps(n)=a(n)$ and $b_\eps(n)=(1-\eps)b_0(n) + \eps b_1(n)$.
If $u_\eps$ is a solution of the difference equation corresponding to $H_\eps$,
then the corresponding Pr\"ufer angles satisfy
\be
\dot{\theta}_\eps(n) = -\frac{W_n(u_\eps, \dot{u}_\eps)}{a(n) \rho_\eps^2(n)},
\ee
where the dot denotes a derivative with respect to $\eps$.

\begin{lemma} \label{prwpsiepsdot}
We have
\be
W_n(s_{\eps,\pm}(z), \dot{s}_{\eps,\pm}(z)) = \begin{cases}
- \sum_{m=n+1}^{N} (b_0(m)-b_1(m)) s_{\eps,+}(z,m)^2,\\
\sum_{m=1}^n (b_0(m)-b_1(m)) s_{\eps,-}(z,m)^2. \end{cases}
\ee
\end{lemma}

\begin{proof}
Summing \eqref{gfb1step} we obtain
\[
W_n(s_{\eps,\pm}(z), s_{\ti{\eps},\pm}(z)) = (\ti{\eps}-\eps) \begin{cases}
- \sum_{m=n+1}^{N} (b_0(m)-b_1(m)) s_{\eps,+}(z,m) s_{\ti{\eps},+}(z,m),\\
\sum_{m=1}^n (b_0(m)-b_1(m)) s_{\eps,-}(z,m) s_{\ti{\eps},-}(z,m). \end{cases}
\]
Now use this to evaluate the limit
\[
\lim_{\ti{\eps} \to \eps}W_n \Big(s_{\eps,\pm}(z),
\frac{s_{\eps,\pm}(z) - s_{\ti{\eps},\pm}(z)}{\eps-\ti{\eps}} \Big).
\]
\end{proof}

Denoting the Pr\"ufer angles of $s_{\eps,\pm}(\lam,n)$ by $\theta_{\eps,\pm}(\lam,n)$,
this result implies for $b_0-b_1\ge 0$,
\begin{align} \nn
\dot{\theta}_{\eps,+}(\lam,n) &= \frac{\sum_{m=n+1}^{N} (b_0(m)-b_1(m)) s_{\eps,+}(z,m)^2}{a(n)\rho_{\eps,+}(\lam,n)^2}
\le 0, \\ \label{thetadot}
\dot{\theta}_{\eps,-}(\lam,n) &= - \frac{\sum_{m=1}^n (b_0(m)-b_1(m)) s_{\eps,-}(z,m)^2}{a(n)\rho_{\eps,-}(\lam,n)^2} \ge 0.
\end{align}
Furthermore, we have the following result from classical perturbation theory.
We add a simple direct proof for convenience of the reader.

\begin{lemma} \label{lemevheps}
Suppose $b_0-b_1\ge 0$ (resp.\ $b_0-b_1\leq 0$).
Then the eigenvalues of $H_\eps$ are analytic functions with respect to
$\eps$ and they are decreasing (resp.\ increasing).
\end{lemma}

\begin{proof}
First of all the Pr\"ufer angles $\theta_{\eps,\pm}(\lam,n)$ are analytic with respect to $\eps$ since
$s_{\eps,\pm}(\lam,n)$ is a polynomial with respect to $\eps$.
Moreover, $\lam\in\sig(H_\eps)$ is equivalent to $\theta_{\eps,+}(\lam,0)\equiv 0 \mod \pi$
(resp.\ $\theta_{\eps,-}(\lam,N) \equiv 0 \mod \pi$) and monotonicity follows from \eqref{thetadot}.
\end{proof}

In particular, this implies that $P(H_\eps)= \# \{E \in \sig(H_\eps) | E < \lam\}$ is continuous
from below (resp. above) in $\eps$ if $b_0-b_1\ge 0$ (resp.\ $b_0-b_1\leq 0$).

Now we are ready for the

\begin{proof}[Proof of Theorem~\ref{thm:main}]
It suffices to prove the result for $\#(s_{0,+}(\lam_0), s_{1,-}(\lam_1))$,
where we can assume $\lam_0=\lam_1=0$ without restriction and set $s_{\eps,\pm}(n)=s_{\eps,\pm}(0,n)$
for notational convenience. We split $b_0-b_1$ according to
$$
b_0 - b_1 = b_+ - b_-,\qquad b_+, b_- \geq 0,
$$
and introduce the operator $H_- = H_0 - b_-$.
Then $H_-$ is a negative perturbation of $H_0$ and
$H_1$ is a positive perturbation of $H_-$.

Furthermore, define $H_\eps$ by
$$
H_\eps = \begin{cases}
H_0 + 2\eps (H_- - H_0), & \eps\in[0,1/2],\\
H_- + 2(\eps -1/2)(H_1 - H_-), & \eps\in[1/2,1].
\end{cases}
$$
Let us look at (using \eqref{wronskinodes})
$$
Q(\eps)=\#(s_{0,+}, s_{\eps,-}) =
\ceil{\Delta_\eps(N)/\pi} - \floor{\Delta_\eps(0)/\pi} -1, \quad
\Delta_\eps(n)=\Delta_{s_{0,+}, s_{\eps,- }}(n)
$$
and consider $\eps\in [0,1/2]$.
At the left boundary $\Delta_\eps(0)$ remains constant whereas at the right
boundary $\Delta_\eps(N)$ is increasing by \eqref{thetadot}.
Moreover, it hits a multiple of $\pi$ whenever $0\in\sig(H_\eps)$.
So $Q(\eps)$ is a piecewise constant function which is continuous from below
and jumps by one whenever $0\in\sig(H_\eps)$. By Lemma~\ref{lemevheps}
the same is true for
$$
P(\eps) = \# \{E \in \sig(H_\eps) | E < 0\} - \# \{E \in \sig(H_0) | E \leq 0\}
$$
and since we have $Q(0)=P(0)$, we conclude $Q(\eps)=P(\eps)$ for all
$\eps\in[0,1/2]$. To see the remaining case $\eps=[1/2,1]$, simply
replace increasing by decreasing and continuous from below by continuous
from above.
\end{proof}

\noindent
{\bf Acknowledgments.}
We thank H. Kr\"uger for several valuable discussions. Furthermore, G.T.\ would like to thank all organizers
of the 14th International Conference on Difference Equations and Applications (ICDEA), Istanbul, July 2008,
and especially Martin Bohner and Mehmet \"Unal, for their kind invitation and the stimulating atmosphere during the meeting.


\begin{thebibliography}{XXXX}
\bibitem{abgr} R. P. Agarwal, M. Bohner, S. R. Grace, D. O'Regan, {\em Discrete Oscillation Theory},
Hindawi Publishing Corp., New York, 2005.
\bibitem{at} F. Atkinson, {\em Discrete and Continuous Boundary Problems},
Academic Press, New York, 1964.
\bibitem{gst} F. Gesztesy, B. Simon, and G. Teschl, {\em Zeros of the Wronskian
and renormalized oscillation Theory}, Am. J. Math. {\bf 118} 571--594 (1996).
\bibitem{kt} H. Kr\"uger and G. Teschl, {\em Relative oscillation theory, weighted zeros of the Wronskian,
and the spectral shift function}, Comm. Math. Phys. {\bf 287:2}, 613--640 (2009).
\bibitem{kt2} H. Kr\"uger and G. Teschl, {\em Relative oscillation theory for Sturm--Liouville
operators extended}, J. Funct. Anal. {\bf 254-6}, 1702--1720 (2008).
\bibitem{kt3} H. Kr\"uger and G. Teschl, {\em Effective Pr\"ufer angles and relative oscillation criteria},
J. Diff. Eq. {\bf 245}, 3823--3848 (2008).
\bibitem{lei} W. Leighton, {\em On self-adjoint differential
equations of second order}, J.~London Math.~Soc. {\bf 27}, 37--47 (1952).
\bibitem{sim} B. Simon, {\em Sturm oscillation and comparison theorems}, in Sturm--Liouville Theory: Past and Present
(eds. W. Amrein, A. Hinz and D. Pearson), 29--43, Birkh\"auser, Basel, 2005.
\bibitem{stu} J.C.F. Sturm, {\em M\'emoire sur les \'equations
diff\'erentielles lin\'eaires du second ordre}, J.~Math.~Pures
Appl., {\bf 1}, 106--186 (1836).
\bibitem{tosc} G. Teschl, {\em Oscillation theory and renormalized oscillation
theory for Jacobi operators}, J. Diff. Eqs. {\bf 129}, 532--558 (1996).
\bibitem{tjac} G. Teschl, {\em Jacobi Operators and Completely Integrable Nonlinear Lattices},
Math. Surv. and Mon. {\bf 72}, Amer. Math. Soc., Rhode Island, 2000.
\end{thebibliography}
\end{document}